\documentclass[10pt,letterpaper,conference]{ieeeconf}

\usepackage[hyperfootnotes=false]{hyperref}
\hypersetup{ 		
    pdfauthor={Tyler Summers},        
    colorlinks=true,        
    linkcolor=NavyBlue,
    urlcolor=NavyBlue,
    citecolor=Mahogany     		
}

\usepackage[pdftex]{graphicx}
\graphicspath{{./fig/}}
\usepackage[usenames,dvipsnames]{color}
\usepackage{epstopdf}
\usepackage{tabularx}
\usepackage{colortbl}
\usepackage{wrapfig}
\usepackage{cite}
\usepackage{booktabs}

\usepackage{amsmath,amssymb,mathrsfs,dsfont}
\usepackage{algorithm}	
\usepackage[noend]{algpseudocode}
\usepackage{todonotes}

\newtheorem{lem}{Lemma}
\newtheorem{thm}{Theorem}
\newtheorem{prop}{Proposition}
\newtheorem{cor}{Corollary}
\newtheorem{defi}{Definition}

\newcommand{\R}{\mathbb{R}}

\newcommand{\m}{\mathcal}

\newcommand{\tr}{\mathbf{tr}}

\IEEEoverridecommandlockouts

\title{Performance guarantees for greedy maximization of \\non-submodular controllability metrics}
\author{Tyler Summers and Maryam Kamgarpour\thanks{T. Summers is with the Department of Mechanical Engineering, University of Texas at Dallas. E-mail: tyler.summers@utdallas.edu. M. Kamgarpour is with the Automatic Control Laboratory, ETH Z\"urich, Switzerland. The work of T. Summers was sponsored by the Army Research Office and was accomplished under Grant Number: W911NF-17-1-0058. The work of M. Kamgarpour is gratefully supported by ERC Starting Grant CONENE.} }

\allowdisplaybreaks
\begin{document}
\maketitle

\begin{abstract}
A key problem in emerging complex cyber-physical networks is the design of information and control topologies, including sensor and actuator selection and communication network design. These problems can be posed as combinatorial set function optimization problems to maximize a dynamic performance metric for the network. Some systems and control metrics feature a property called \emph{submodularity}, which allows simple greedy algorithms to obtain provably near-optimal topology designs. However, many important metrics lack submodularity and therefore lack  provable guarantees for using a greedy optimization approach. Here we show that performance guarantees can be obtained for greedy maximization of certain \emph{non-submodular} functions of the controllability and observability Gramians. Our results are based on two key quantities: the \emph{submodularity ratio}, which quantifies how far a set function is from being submodular, and the \emph{curvature}, which quantifies how far a set function is from being supermodular. 
\end{abstract}

\section{Introduction}
Many emerging complex dynamical networks, from critical infrastructure to industrial cyber-physical systems and biological networks, are increasingly able to be instrumented with new sensing, actuation, and communication capabilities. These networks comprise growing webs of interconnected feedback loops and must operate efficiently and resiliently in dynamic and uncertain environments. This motivates addressing fundamental network topology design problems to select the most effective sensors, actuators, and communication links, along with jointly designing the associated estimation, control, and communication policies.

There are a variety of quantitative notions of network controllability and observability to guide topology design in cyber-physical networks. Examples include Kalman rank conditions \cite{liu2011controllability,nepusz2012controlling,ruths2014control,olshevsky2014minimal,pequito2016}, controllability and observability Gramians \cite{Pasqualetti2014c,summers2014optimal,summers2014submodularity,tzoumas2016,yannature2015,zhao2016scheduling,nozari2016time}, and optimal and robust feedback control and estimation performance metrics \cite{Polyak-LMI_sparse_fb,Dhingra2014,munz2014sensor,Argha2016,summers2016actuator,zhang2017sensor,Taylor2017,Taha2017}. Topology design problems have also been considered for specific classes of networks, including leader selection and communication network design \cite{clark2016submodularity,summers2017information}. Various optimization methods have been proposed for topology design, including greedy algorithms \cite{tzoumas2016,zhang2017sensor,summers2014submodularity,summers2016actuator,clark2016submodularity}, convex relaxation heuristics with sparsity inducing regularization \cite{Polyak-LMI_sparse_fb,Dhingra2014,munz2014sensor,Argha2016,summers2016convex}, and mixed-integer semidefinite programming methods\cite{Taylor2017,Taha2017}. These methods are all heuristic approximations to extremely difficult combinatorial optimization problems.

A key combinatorial property in network topology design problems is \emph{submodularity}; simple greedy algorithms have theoretical performance guarantees for submodular set function maximization problems. Submodularity has recently been discovered in several network topology design problems in systems and control, including  certain Gramian metrics and leader selection problems~\cite{summers2014submodularity,tzoumas2016,clark2016submodularity}, paralleling similar development for information optimization problems in machine learning. However, many other important topology design problems lack submodularity \cite{summers2016actuator,summers2017information}. Nevertheless, greedy algorithms were shown to be highly effective empirically for many non-submodular optimal control problems, despite the apparent lack of theoretical guarantees.

Here, we show that theoretical performance guarantees for greedy algorithms can be obtained for certain \emph{non-submodular} set function optimization problems in systems and control. Our results are based upon recent fundamental work \cite{bian2017guarantees} that generalizes classical optimality bounds for submodular set functions. This work utilizes two key quantities: the \emph{submodularity ratio}, which quantifies how far a set function is from being submodular, and the \emph{curvature}, which quantifies how far a set function is from being modular. We focus on two \emph{non-submodular} Gramian-based metrics for controllability and observability: the minimum eigenvalue and negative trace of the inverse. Specifically, we derive general bounds on the submodularity ratio and the curvature of these two functions based on eigenvalue inequalities for sums of symmetric matrices. The existence of these bounds support the use of the greedy algorithm beyond a heuristic for a much wider class of network topology design problems. 

Our preliminary results appeared in an unpublished manuscript in \cite{summers2017kamgarpour}. Afterwards, two recent  works also utilized approximate submodularity and supermodularity notions for network control design problems \cite{gupta2018approximate,tzoumas2018control}. Our work has been developed independently and prior to the above and our approach complements them as follows. First, \cite{gupta2018approximate} derive bounds by comparison of the non-submodular functions with a ``close" submodular one. Second, \cite{tzoumas2018control} considers an objective function different than ours, arising in co-design of control and estimation for a network system. 

The rest of the paper is organized as follows. Section II provides preliminaries on set function optimization and Gramian-based controllability metrics. Section III develops our results on bounds on the submodularity ratio and curvature for certain non-submodular functions of the controllability Gramian, leading to performance guarantees from the greedy algorithm. Section IV discusses the predictive power of the theory with case studies. We conclude in Section V.

\section{Preliminaries}

\subsection{Set functions and submodularity}
Network topology design problems can be formulated as cardinality constrained set function optimization problems
\begin{equation} \label{optprob}
 \underset{{S \subseteq V, \ |S| \leq k }}{\text{maximize}} \quad f(S),
\end{equation}
 where $V = \{1,...,M \}$ is a finite set,  $f: 2^V \rightarrow \mathbf{R}$ is a set function that maps each subset of $V$ to a real number, and $k$ denotes a fixed number of elements to be selected from $V$. These problems are combinatorial and finite, and thus can be solved in principle  by exhaustive search. However, this approach quickly becomes intractable even for moderately sized problems. The motivating context of large cyber-physical networks requires a different approach.

Greedy algorithms are a simple alternative to exhaustive search. The greedy algorithm for set function maximization is shown in Algorithm \ref{algorithm:greedy}. 
\begin{algorithm}[b] 
\caption{The greedy algorithm.}
\label{algorithm:greedy}
\begin{algorithmic}
\State $S \leftarrow \emptyset$
\While {$|S| \leq k$}
\State $e^\star = \mathop{\mbox{argmax}}\limits_{e\in V\setminus S } \quad f(S\cup\{e\}) - f(S)$
\State $S \leftarrow S \cup \{e^\star\}$
\EndWhile
\State $S^\star \leftarrow S$
\end{algorithmic}
\end{algorithm}
When a set function maximization problem has a certain property called \emph{submodularity}, the greedy algorithm achieves results that are provably within a constant factor of the optimal value. 
\begin{defi} \label{submoddef}
A set function $f: 2^V \rightarrow \mathbf{R}$ is called \emph{submodular} if for all subsets $A \subseteq B \subseteq V$ and all elements $s \notin B$, it holds that
\begin{equation} \label{submod1}
f(A \cup \{s\}) - f(A) \geq f(B \cup \{s\}) - f(B).
\end{equation}
A set function is \emph{supermodular} if the reversed inequality in (\ref{submod1}) holds and is \emph{modular} if  (\ref{submod1}) holds with equality. A set function is \emph{monotone nondecreasing} if $\forall A \subset V, s \in V$, $f(A \cup \{s \}) \geq f(A)$. A set function is  \emph{normalized} if $f(\emptyset) = 0$. 
\end{defi}

Intuitively, submodularity is a diminishing returns property, that is, adding an element to a smaller set gives a larger benefit than adding it to a larger set. This intuition is utilized and quantified to derive constant-factor approximation guarantees for the greedy algorithm applied to submodular maximization problems subject to cardinality constraints.

\begin{thm}[Nemhauser 1978 \cite{nemhauser1978analysis}] \label{greedyboundthm}
Let $f^*$ be the optimal value of the set function optimization problem \eqref{optprob}, and let $f(S_{greedy})$ be the value associated with the subset $S_{greedy}$ obtained from applying the greedy algorithm on \eqref{optprob}. If $f$ is submodular nondecreasing, then
\begin{equation}
	f(S_{greedy}) \geq \left(1 - \frac{1}{e} \right) f^*
\label{eq:greedy_bound}\end{equation}
\end{thm}
The above theorem has rendered the greedy approach an algorithm of choice for several challenging combinatorial optimization problems\footnote{This bound can be refined by explicitly considering the number of elements $k$ to be chosen or including further properties of  $f$.}.

Several problems in systems and control that feature sub- or supermodularity have been recently explored \cite{summers2014submodularity,clark2016submodularity,tzoumas2016}. However, a large class of important set function optimization in network topology design fail to be sub- or supermodular \cite{summers2016actuator,summers2017submodularitycorrection,summers2017information}. It has been observed that by quantifying how close a function to  being sub- or supermodular is, one can derive constant factor optimization for the greedy approach. This ``closeness" is with respect to two notions defined below. 

Let $\rho_A(B) := f(A \cup B)- f(B)$ denote the marginal benefit of the set $A \subset V$ with respect to the set $B\subset V$. For notational compactness, we use $\omega$ interchangeably with $\{\omega\}$ when considering a singleton subset of $V$.
\begin{defi} \label{sratio}
The \emph{submodularity ratio} of a nonnegative set function $f$ is the largest $\gamma \in \R_+$ such that 
\begin{align}
\label{def:gamma}
\sum_{\omega \in \Omega \setminus S} \rho_\omega(S) \, \geq \, \gamma \rho_\Omega(S),
 \quad \forall \Omega, S \subseteq V.
\end{align} 
\end{defi}
\begin{defi} \label{curvature}
The \emph{curvature} of a nonnegative set function $f$ is the smallest $\alpha \in \R_+$ such that 
\begin{align}
\label{def:alpha}
 \rho_j(S\setminus j \cup \Omega) \, \geq\, &(1-\alpha) \rho_j(S\setminus j), \forall \Omega, S \subseteq V, \forall j \in S \setminus \Omega.
\end{align} 
\end{defi}
For a nondecreasing function $\gamma \in [0,1]$, and $\gamma = 1$ if and only if $f$ is submodular. The  curvature $\alpha$ of a nondecreasing function is contained in $[0,1]$, and $\alpha = 0$ if and only if $f$ is supermodular \cite{bian2017guarantees}. The following recent result \cite{bian2017guarantees} generalizes Theorem \ref{greedyboundthm} and  provides performance guarantees for greedy maximization of \emph{non-submodular} functions based on the submodularity ratio and  curvature.
\begin{thm}[\cite{bian2017guarantees}] \label{nonsubmodthm}
Let $f$ be a nonnegative nondecreasing normalized set function with submodularity ratio $\gamma \in [0,1]$ and curvature $\alpha \in [0,1]$. Then, Algorithm \ref{algorithm:greedy} on problem \eqref{optprob} enjoys the following approximation guarantee:
$$ f(S_{greedy}) \geq \frac{1}{\alpha} \left( 1 - e^{-\alpha \gamma} \right) f^*.$$
\end{thm}

It is intractable to compute the submodularity ratio and curvature for a given set function due to the combinatorial number of constraints (of order of $2^{ 2 | V |}$) in \eqref{def:gamma} and \eqref{def:alpha}, respectively (similar to the challenge in exhaustive search for solving Problem \eqref{optprob}). {However,  a positive lowerbound on the submodularity ratio and an upperbound on the curvature for a given $f$, justify the use of the greedy algorithm for Problem \eqref{optprob} via Theorem \ref{nonsubmodthm}}. Our goal is to derive such bounds for $f$'s corresponding to non-submodular controllability metrics.

\subsection{Gramian-based performance metrics}
Consider the linear  system describing network dynamics
\begin{equation} \label{dtdlds}
\dot{x}(t) = A x(t) + B_S u(t), \quad t = 0,...,T,
\end{equation}
where $x(t) \in \mathbf{R}^n$ is the network state at time $t$, $u(t) \in \mathbf{R}^{ |S| }$ is the input at time $t$. Let $V = \{ b_1,..., b_M \}$ be a finite set of $n$-dimensional column vectors associated with possible locations for actuators that could be placed in the system, i.e., for  $S \subset V$, the input matrix is $b_S = [b_{s_1}, ..., b_{s_{|S|}}] \in \mathbf{R}^{n \times |S|}$. 
We focus on the network topology design problem of selecting a set of actuators to optimize certain metrics of the controllability Gramian. (Analogous results follow for sensor selection based on the observability Gramian.) 

The infinite-horizon controllability Gramian associated with a subset $S\subset V$ of actuators is the symmetric positive semidefinite matrix  $W_S$, satisfying
\begin{equation} \label{gramint}
W_S = \int_0^\infty e^{A \tau} b_S b_S^T e^{A^T \tau} d\tau \ \in \mathbf{R}^{n\times n}.
\end{equation}
To ensure the Gramian is well-defined we assume $A$ is stable. 
Observe that $W_{S \cup \Omega} = W_{S} + W_{\Omega \setminus S}$, $\forall S, \Omega \subset V$. This additive dependence on the actuators  is key in deriving  properties of several performance metrics based on the Gramian. 

The quantity $x^T W_S^{-1} x$ is the amount of input energy required to transfer the network state from the origin to the state $x$.  As such, we can define the following scalar metrics of the matrix $W_S$, each of which defines a different set function that provides a basis for actuator selection. 
\begin{itemize}
\item $f(S) = \mathbf{tr}(W_S)$ (modular): This metric is inversely related to the average input energy and can be interpreted as the average controllability in all directions in the state space. It also quantifies the system $\mathcal{H}_2$ norm.
\item$f(S) = \log \det W_S$ (submodular): This  is a volumetric measure of the set of states that
can be reached with one unit of input energy. Note that even if trace of the Gramian is large the volume can be small, due to  reachability only in certain state space directions. 
\item $f(S) = \mathbf{rank}(W_S)$ (submodular): This metric captures the dimension of controllable subspace.
\item $f(S) = \lambda_{min} (W_S)$ (\emph{not} submodular): This metric captures the energy for  directions that are hard to control. 
\item $f(S) = -\mathbf{tr}(W_S^{-1})$ (\emph{not} submodular): This metric captures the average energy required to reach any arbitrary direction of the state space. 
\end{itemize}
The choice of a specific metric from above is  dependent on the application at hand (see \cite{summers2014submodularity} for discussions). Along with each choice, one is  faced with the problem of selecting actuators to optimize the metric. This selection is a combinatorial problem and hence, quickly becomes intractable as the network size increases. Past work used submodularity to support empirical evidence on effectiveness of the greedy algorithm for the first three metrics \cite{summers2014submodularity,summers2017submodularitycorrection,clark2017submodularity}. It was shown that the last two metrics  $\lambda_{min} (W_S)$ and $-\mathbf{tr}(W_S^{-1})$, are \emph{not} submodular \cite{summers2014submodularity,summers2017submodularitycorrection}. These two metrics are important since they refer to the worst-case and the average energy respectively, required to steer the system. 
However, the empirical performance  of the greedy optimization were not supported theoretically for these cases. 

\section{Performance guarantees for non-submodular Gramian metrics}
To support the use of greedy algorithm beyond a heuristic approach for maximizing $-\mathbf{tr}(W_S^{-1})$ and $\lambda_{min} (W_S)$, we will derive positive lowerbounds on the submodularity ratio and upperbounds on the curvature of these functions. 
Let $\m{S}^n, \m{S}^n_+, \m{S}^n_{++}$ denote the set of $n$-dimensional symmetric, symmetric positive semidefinite, and symmetric positive definite matrices, respectively. For $M \in \m{S}^n$, let $\lambda_1(M) \geq \lambda_2(M) \geq \dots \geq \lambda_n(M) $ denote the eigenvalues of the matrix $M$. Let $I_n$ denote the identity matrix of dimension $n$. 

Let us recall Weyl's inequalities for eigenvalues of sum of two symmetric matrices \cite{weyl1912asymptotische}.
\begin{lem} 
\label{Weyl}
Let $A, B \in \m{S}^n$. 
\begin{align}
\label{eq:evals_sum}
\lambda_i(A) + \lambda_n(B) \leq \lambda_i(A + B) \leq \lambda_i(A) + \lambda_1(B). 
\end{align}
\end{lem}

\subsection{Minimum eigenvalue of the Gramian}
\begin{prop}
\label{lambda_min}
The function $f(S)=\lambda_{n}(W_S)$ is nonnegative  nondecreasing and normalized. The submodularity ratio $\gamma$, and the curvature $\alpha$, of this function  are bounded by
$$ \gamma \geq \frac{ \min_{\omega \in V} \lambda_n(W_\omega)}{ \max_{\omega \in V} \lambda_1(W_\omega)}, \quad \alpha \leq 1-\frac{ \min_{\omega \in V} \lambda_n(W_\omega)}{ \max_{\omega \in V} \lambda_1(W_\omega)}.$$ 
\end{prop}

\begin{proof}
The nonnegativity of $\lambda_{n}(W_S)$ follows from the Gramian being positive semidefinite. The fact that it is nondecreasing follows from the eigenvalue inequality in \eqref{eq:evals_sum} and from nonnegativity of $\lambda_n(W_S)$. In particular, note that for  $\forall S \subset V$, $\forall \omega \in V$, $W_{S \cup \omega} = W_S + W_{\omega}$. Hence, $\lambda_n(W_{S \cup \omega}) \geq \lambda_n(W_S) + \lambda_n(W_{\omega}) \geq \lambda_n(W_S)$. It is easy to see that $\lambda_n(W_\emptyset) = 0$ and hence this function is normalized. 

To bound the submodularity ratio, we first derive a lowerbound for the left-hand-side in \eqref{def:gamma}. Note that $\forall S, \Omega \subset V$
\begin{align*}
\sum_{\omega \in \Omega \setminus S} \rho_\omega(S) =&\sum_{\omega \in \Omega \setminus S} f(S \cup \{\omega\})- f(S)\\
=&\sum_{\omega \in \Omega \setminus S} \lambda_n(W_{S \cup \omega})-\lambda_n(W_S ) \\
 \geq&  \sum_{\omega \in \Omega \setminus S}  \lambda_n(W_S) + \lambda_n(W_\omega) - \lambda_n(W_S)\\
=&\sum_{\omega \in \Omega \setminus S} \lambda_n(W_{\omega})
\geq | \Omega\setminus S | \min_{\omega \in V} \lambda_n(W_\omega),
\end{align*}
where in the first inequality we used Lemma \ref{Weyl}. Next, for the right-hand-side we have $\forall S, \Omega \subset V$
\begin{align*}
 \rho_\Omega(S) =&f(S \cup \Omega)- f(S)\\
=&\lambda_n(W_{S \cup \Omega})-\lambda_n(W_S ) \\
\leq& \lambda_n(W_S)  + \lambda_1(W_{\Omega\setminus S}) - \lambda_n(W_S)
\\=& \lambda_1(W_{\Omega\setminus S}) 
\leq | \Omega\setminus S | \max_{\omega \in V} \lambda_1(W_\omega),
\end{align*}
where in the first inequality we also used Lemma \ref{Weyl}. Putting the above two inequalities together yields 
\begin{align*}
\gamma \geq \frac{ \min_{\omega \in V} \lambda_n(W_\omega)}{ \max_{\omega \in V} \lambda_1(W_\omega)}.
\end{align*}

We similarly bound the curvature. Note that for the left-hand-side in \eqref{def:alpha}, we have $\forall S, \Omega \subset V, j \in S \setminus \Omega$
\begin{align*}
 \rho_j(S\setminus j \; \cup \Omega) =&\lambda_n(W_{S \cup \Omega})-\lambda_n(W_{S\setminus j \;  \cup \Omega }) \\
\geq& \lambda_n(W_{S\setminus j \;  \cup \Omega })   + \lambda_n(W_{j}) - \lambda_n(W_{S\setminus j \;  \cup \Omega })\\
=& \lambda_n(W_{j}) 
\geq  \min_{j \in V} \lambda_n(W_i).
\end{align*}
Next, for the right-hand-side we have $\forall S, \Omega \subset V, j \in S \setminus \Omega$
\begin{align*}
\rho_j(S\setminus j)&= \lambda_n(W_{S})-\lambda_n(W_{S\setminus j }) \\
\leq & \lambda_n(W_{S\setminus j }) + \lambda_1(W_j) -\lambda_n(W_{S\setminus j })\\
 \leq & \max_{j \in V}\lambda_1(W_j).
\end{align*}
Putting the above two inequalities together, we have
\begin{align*}
\frac{\rho_j(S\setminus j \; \cup \Omega) }{ \rho_j(S\setminus j)} \geq \frac{ \min_{\omega \in V} \lambda_n(W_\omega)}{ \max_{\omega \in V} \lambda_1(W_\omega)}, \ \   \forall S, \Omega \subset V, j \in S \setminus \Omega.  
\end{align*}
Hence, we obtain the claimed lowerbound for the curvature
\begin{align*}
\alpha \leq 1-\frac{ \min_{\omega \in V} \lambda_n(W_\omega)}{ \max_{\omega \in V} \lambda_1(W_\omega)}.
\end{align*}
\end{proof}

\textbf{Remark:} The above bound is useful in deriving performance guarantees via Theorem \ref{nonsubmodthm} only if $\min_{\omega \in V} \lambda_{min}(W_\omega) > 0$. This condition is equivalent to requiring that each actuator results in controllability of the network - an unreasonable requirement for large-scale sparse networks. A potential  approach to get around this issue  is to assume a set of existing actuators $B_0 \in \R^{n \times m}$ render the system controllable and the goal is to choose additional actuators to improve the controllability. Unfortunately, even with this assumption we face the same difficulty. In particular, note that in this case, the Gramian would be $\bar{W}_S = W_{B_0 \cup S}$. Hence, $\bar{W}_{S \cup \omega} = \bar{W}_S + W_{\omega}$ and $\lambda_n(\bar{W}_{S \cup \omega} ) - \lambda_n(\bar{W}_S) \geq \lambda_n(W_\omega)$. Since $W_\omega$ does not contain $B_0$ in general, its minimum eigenvalue would be zero and the same problem arises?. The bound  in the recent work of \cite{gupta2018approximate} on the closeness of the minimum eigenvalue of the Gramian to a submodular function, seems to also suffer from the same drawback. Despite this theoretical limitation, our empirical estimates of the submodularity ratio and curvature, computed in the numerical section, for all  problem instances considered were bounded (far) away from zero and one, respectively. 
As we will shortly see, in the metric described by inverse of the trace of the Gramian, we can avoid vacuous bounds by including $B_0$ corresponding to a set of base actuators ensuring controllability. 

\subsection{Trace of inverse of Gramian}
The function $-\tr({W_S}^{-1})$ is not well-defined if $W_S$ is not invertible. To avoid this issue,  we  assume that an existing set of actuators corresponding to $B_0$ provides controllability (considered also implicitly in \cite{tzoumas2016minimal,gupta2018approximate}). Hence, our input set is given by $B_S = [B_0, b_{s_1}, ..., b_{s_{|S|}}] \in \mathbf{R}^{n \times (m + |S|)}$. We use $\bar{W}_S$ for the Gramian corresponding to the augmented set. 

\begin{prop} \label{traceinvbound}
The set function $f(S)=-\tr({\bar{W}_S}^{-1})$ is monotone nondecreasing. Its  submodularity ratio $\gamma$, and  curvature $\alpha$, are bounded by
\begin{align}
\label{trace_lambda}
1 > \gamma \geq \frac{\min_{\omega \in V} \tr({W}_\omega)\big(\min_{\omega \in V}\lambda_n(  \bar{W}_\omega)\big)^2}{\max_{\omega \in V} \tr(W_\omega)\big(\lambda_1(  \bar{W}_V)\big)^2} > 0, \\
\label{trace_alpha}
0 < \alpha \leq 1-\frac{\min_{\omega \in V} \tr(W_\omega)\big( \min_{\omega \in V}\lambda_n(  \bar{W}_\omega)\big)^2}{\max_{\omega \in V} \tr(W_\omega)\big(\ \lambda_1(  \bar{W}_V)\big)^2} < 1.
\end{align}
\end{prop}

\begin{proof}
First, we show that $f(S)$ is nondecreasing. For all $\omega \in V, S \subset V$, we have
\begin{align}
\nonumber
&f(S \cup \{\omega\})- f(S)= -\tr\big(({\bar{W}_{S\cup \omega}})^{-1}\big) + \tr\big((\bar{W}_S)^{-1}\big)\\
\nonumber
=&\sum_{i=1}^n -\lambda_i\big((\bar{W}_{S\cup \omega})^{-1}\big) + \lambda_i\big((\bar{W}_S)^{-1}\big)\\
\nonumber
=&\sum_{i=1}^n \frac{-1}{\lambda_i\big(\bar{W}_{S\cup \omega}\big)} + \frac{1}{\lambda_{i}\big(\bar{W}_S\big)}\\
\nonumber
=&\sum_{i=1}^n \frac{  \lambda_i(  \bar{W}_{S\cup \omega})-\lambda_{i}( \bar{W}_S) }{\lambda_i(  \bar{W}_{S\cup \omega})\lambda_{i}( \bar{W}_S)} \\
\nonumber
\geq&\frac{\sum_{i=1}^n  \lambda_i(  \bar{W}_{S\cup \omega})-\lambda_{i}( \bar{W}_S)}{\lambda_1(  \bar{W}_V)\lambda_{1}( \bar{W}_V)} \\
\label{eq:bound_monotone}
=& \frac{\tr(\bar{W}_{S\cup \omega})-\tr( \bar{W}_S) }{\big( \lambda_1(  \bar{W}_V)\big)^2} \geq 0.
\end{align}

To derive a lowerbound for the submodularity ratio \eqref{def:gamma}, we first derive a lowerbound for the left-hand-side in \eqref{def:gamma}.
\begin{align}
\nonumber
&\sum_{\omega \in \Omega \setminus S} f(S \cup \{\omega\})- f(S)\\
\nonumber
=&\sum_{\omega \in \Omega \setminus S}  -\tr\big((\bar{W}_{S\cup \omega})^{-1}\big) + \tr\big((\bar{W}_S)^{-1}\big) \\
\nonumber
\geq&\sum_{\omega \in \Omega \setminus S} \frac{\tr(\bar{W}_{S\cup \omega})-\tr( \bar{W}_S) }{\big( \lambda_1(  \bar{W}_V)\big)^2} \\
\label{lowerbound_trace}
\geq& | \Omega \setminus S| \frac{\min_{\omega \in V} \tr(W_\omega)}{\big(\lambda_1(  \bar{W}_V)\big)^2}.
\end{align}
To get the second inequality above we summed \eqref{eq:bound_monotone}  over $\omega \in \Omega \setminus S$. The last inequality trivially follows.  Notice that the numerator above is greater than zero since the Gramian has at least one positive eigenvalue. Next, we upper  bound the right-hand-side in \eqref{def:gamma}.
\begin{align}
\nonumber
&f(S \cup \{\Omega\})- f(S)\\
\nonumber
=& -\tr\big(( \bar{W}_{S\cup \Omega})^{-1}\big) + \tr\big(( \bar{W}_S)^{-1}\big) \\
\nonumber
=& \sum_{i=1}^n \frac{-1}{\lambda_{i}\big(\bar{W}_{S\cup \Omega}\big)} + \frac{1}{\lambda_{i}\big( \bar{W}_S\big)}\\
\nonumber
=&\sum_{i=1}^n \frac{\lambda_{i}(  \bar{W}_{S\cup \Omega})-\lambda_{i}( \bar{W}_S) }{\lambda_i(  \bar{W}_{S\cup \Omega})\lambda_{i}( \bar{W}_S)} \\
\nonumber
\leq&\frac{\sum_{i=1}^n  \big(\lambda_i(  \bar{W}_{S\cup \Omega})-\lambda_{i}( \bar{W}_S) \big)}{\big(\min_{\omega \in V}\lambda_n(  \bar{W}_\omega)\big) \big( \min_{\omega \in V}\lambda_n(  \bar{W}_\omega)\big)} \\
\nonumber
=& \frac{\tr( \bar{W}_{S\cup \Omega})-\tr( \bar{W}_S) }{\big(\ \min_{\omega \in V}\lambda_n(  \bar{W}_\omega)\big)^2} =\sum_{\omega \in \Omega \setminus S} \frac{\tr({W}_{\omega})}{\big( \min_{\omega \in V}\lambda_n(  \bar{W}_\omega)\big)^2} \\
\label{upperbound_trace}
\leq& | \Omega \setminus S| \frac{\max_{\omega \in V} \tr(W_\omega)}{\big( \min_{\omega \in V}\lambda_n(  \bar{W}_\omega)\big)^2} .
\end{align}
Notice that the denominator above is greater than zero since  the assumption of having a set of existing actuators providing controllability, implies $\bar{W}_{\emptyset} = W_{B_0} $ is positive definite. Hence,   $\lambda_n(\bar{W}_\omega) = \lambda_n(W_{B_0 \cup \omega})> 0$ for all $\omega \in V$ in the second to last inequality above. 
Putting  inequalities \eqref{lowerbound_trace} and \eqref{upperbound_trace} we obtain the positive lowerbound on $\gamma$ as in \eqref{trace_lambda}.

Our technique for deriving a bound on curvature is very similar to that of the submodularity ratio. In particular, we first derive a lowerbound for the left-hand-side in \eqref{def:alpha}.
\begin{align*}
& \rho_j(S\setminus j \; \cup \Omega) =f(S \cup \Omega)- f(S\setminus j \;  \cup \Omega)\\
=& -\tr\big(( \bar{W}_{S\cup \Omega})^{-1}\big) + \tr\big((\bar{W}_{S\setminus j\;  \cup \Omega})^{-1}\big) \\
=&\sum_{i=1}^n \frac{-1}{\lambda_i\big(\bar{W}_{S\cup \Omega}\big)} + \frac{1}{\lambda_{i}\big(\bar{W}_{S\setminus j\;  \cup \Omega}\big)}\\
=&\sum_{i=1}^n \frac{\big(\lambda_i(  \bar{W}_{S\cup \Omega})-\lambda_{i}( \bar{W}_{S\setminus j\;  \cup \Omega}) \big)}{ \lambda_i(  \bar{W}_{S\cup \Omega})\lambda_{i}( \bar{W}_{S\setminus j\;  \cup \Omega})} \\
\geq& \frac{\sum_{i=1}^n  \big(\lambda_i(  \bar{W}_{S\cup \Omega})-\lambda_{i}( \bar{W}_{S\setminus j\;  \cup \Omega}) \big)}{ \lambda_1(  \bar{W}_V)\lambda_{1}( \bar{W}_V)} \\
=& \frac{\tr(\bar{W}_{S\cup \Omega})-\tr( \bar{W}_{S\setminus j\;  \cup \Omega}) }{\big( \lambda_1(  \bar{W}_V)\big)^2}\\
 = &\frac{\tr({W}_{j})}{\big( \lambda_1(  \bar{W}_V)\big)^2}
 \geq  \frac{\min_{\omega \in V} \tr(W_\omega)}{\big( \lambda_1(  \bar{W}_V)\big)^2}.
\end{align*}
Next, we upper  bound the right-hand-side in \eqref{def:alpha} as follows.
\begin{align*}
&\rho_j(S\setminus j)=f(S)-f(S\setminus j)\\
=&-\tr\big((\bar{W}_{S})^{-1}\big) + \tr\big((\bar{W}_{S\setminus j})^{-1}\big) \\
=& \sum_{i=1}^n \frac{-1}{\lambda_{i}( \bar{W}_{S})} + \frac{1}{\lambda_{i}(\bar{W}_{S\setminus j})}\\
=&\sum_{i=1}^n \frac{\lambda_{i}(  \bar{W}_{S})-\lambda_{i}( \bar{W}_{S\setminus j})}{\ \lambda_i(  \bar{W}_{S})\lambda_{i}( \bar{W}_{S\setminus j})} \\
\leq& \frac{\sum_{i=1}^n \lambda_i(  \bar{W}_{S})-\lambda_{i}( \bar{W}_{S\setminus j}) }{\big(\min_{\omega \in V}\lambda_n(  \bar{W}_\omega)\big)^2}  
= \frac{\tr( \bar{W}_{S})-\tr( \bar{W}_{S\setminus j}) }{\big( \min_{\omega \in V}\lambda_n(  \bar{W}_\omega)\big)^2}\\
 =&\frac{\tr(  W_{j})}{\big( \min_{\omega \in V}\lambda_n(  \bar{W}_\omega)\big)^2}
 \leq  \frac{\max_{\omega \in V} \tr(W_\omega)}{\big( \min_{\omega \in V}\lambda_n(  \bar{W}_\omega)\big)^2} .
\end{align*}
Putting the above  lowerbound and upperbound together, from the curvature definition in \eqref{curvature}  we obtain the upperbound on the curvature given in \eqref{trace_alpha}. 
\end{proof}

Since $f(\emptyset) =-\tr( (\bar{W}_\emptyset)^{-1}) = -\tr( (W_{B_0})^{-1}) \neq 0$, $f$ is not normalized. This offset appears in the suboptimality gap of the greedy algorithm through Theorem \ref{nonsubmodthm} as follows. 
\begin{cor}
\label{cor:non_norm}
Consider the function $f(S) = -\tr((\bar{W}_S)^{-1})$ with  submodularity ratio $\gamma$ and curvature  $\alpha$. 
$$f(S_{greedy}) -f(\emptyset)  \geq \frac{1}{\alpha} \left( 1 - e^{-\alpha \gamma} \right) \big(f^* - f(\emptyset)\big).$$
\end{cor}
\begin{proof}
Consider the  function $\bar{f}(S) =f(S)- f(\emptyset) = - \tr( \bar{W}_{S}^{-1}) + \tr(W_{B_0}^{-1})$. It is normalized, i.e. $\bar{f}(\emptyset) = 0$.  From Proposition \ref{traceinvbound}, $f(S)$ is monotone nondecreasing and this implies $\bar{f}$ is nonnegative monotone nondecreasing. Since the submodularity ratio and the curvature remain invariant if a constant term is added to the function $f$, these parameters for $\bar{f}$ are equivalent to those of $f$. Hence, applying Theorem \ref{nonsubmodthm} to $\bar{f}$ gives the claim of the Corollary.
\end{proof}

\textbf{Remark:} 
You might observe that in the considered instances of deriving the upperbound on the curvature, $\alpha^u$, we obtained  $\alpha^u = 1-\gamma^l$, where $\gamma^l$ was our lowerbound on the submodularity ratio. This connection is due to the (conservative) approach in deriving these bounds. Such connection does not exist for the true values of the submodularity ratio and the curvature in general. In particular, for a submodular function $\gamma = 1$. This clearly would not imply that $f$'s curvature is $\alpha \leq 1-1 =0$; otherwise, $f$ would also be supermodular and hence, modular.  

The above bounds are valid for any problem instance -- any stable network dynamics matrix  $A$ and  set of possible actuator locations $B_S$. Hence, these bounds are often conservative as will be seen in the next section. Nevertheless, their existence is promising as it supports empirical observations about effectiveness of the greedy algorithm  for optimizing non-submodular set functions related to the Gramian.


\section{Case Studies}
\begin{figure}[t]
\begin{center}
\includegraphics[scale=0.36]{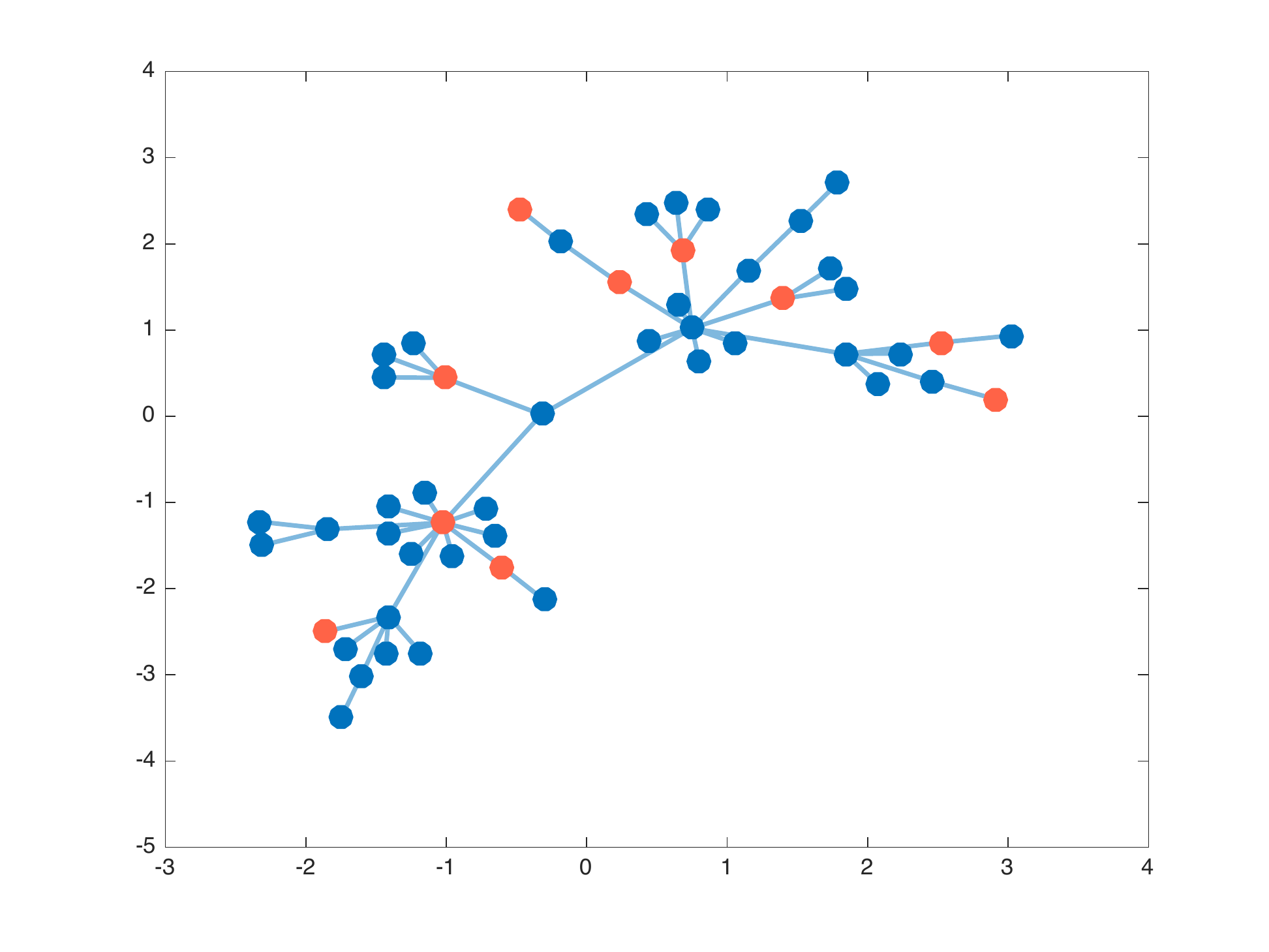}
\end{center}
\caption{Barab{\'a}si-Albert network on $n=50$ nodes, with $k=10$ nodes selected using the greedy algorithm with the Gramian trace inverse metric to receive control inputs (shown in red).} 
\label{fig:barabasi}
\end{figure}

\begin{figure}[t]
\begin{center}
\includegraphics[scale=0.36]{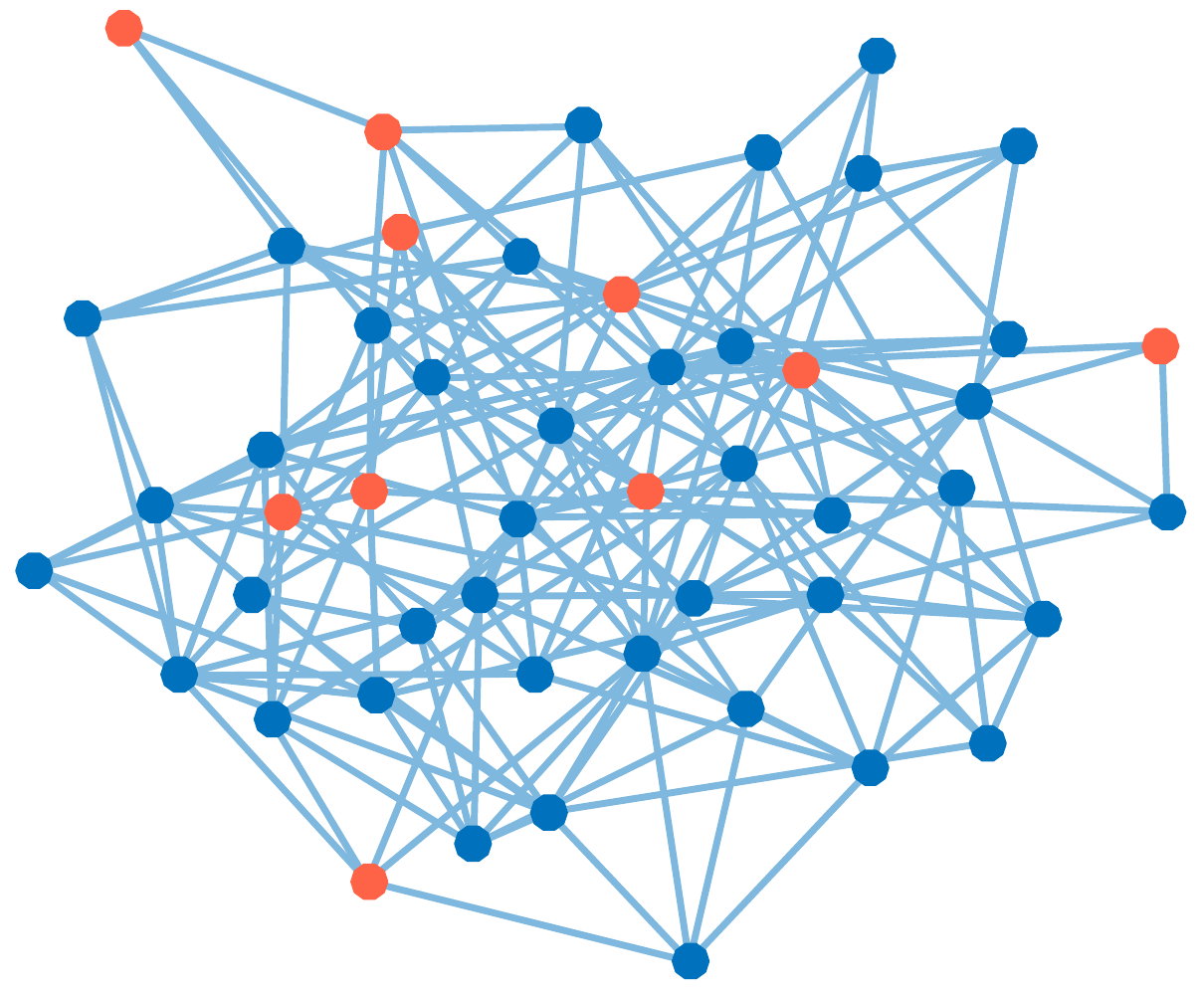}
\end{center}
\caption{Erd{\H{o}}s-R{\'e}nyi network on $n=50$ nodes, with $k=10$ nodes selected using the greedy algorithm with the Gramian trace inverse metric to receive control inputs (shown in red).} 
\label{fig:erdos}
\end{figure}

We present several illustrative numerical examples in random networks. In all examples, we first generate unweighted random graphs whose structure defines non-zero entries in the network dynamics matrix as discussed below. We then randomly generate edge weights associated with non-zero entries by drawing independently from a standard normal distribution. Finally, we  shift the matrix so that it is stable, with the smallest eigenvalue(s) having real part $-0.05$. 
Figure \ref{fig:barabasi} shows an instance of a Barab{\'a}si-Albert network, whose edge structure is generated with a preferential attachment mechanism that produces power law degree distributions \cite{barabasi1999emergence}. This network connectivity is motivated by link formations in social networks. Figure \ref{fig:erdos} shows an instance of an Erd{\H{o}}s-R{\'e}nyi network, with the edge probability chosen uniformly for all nodes to be 0.08  (just above the critical value of $\ln(50)/50$ to ensure connectivity of the network \cite{erdHos1959random}). In both cases, there are $n=50$ states and $k=10$ possible inputs he selection of $k=10$ inputs. We assume that an input signal could be injected into any state node, i.e., the set V corresponds to the standard basis in $\mathbf{R}^n$. The selected nodes using the greedy algorithm to minimize the trace of the controllability Gramian inverse are shown in Figures \ref{fig:barabasi}-\ref{fig:erdos}. 

In  the above problems, computing exact bounds for the submodularity ratio and curvature are prohibitive due to the network sizes. We empirically estimated  the corresponding submodularity ratio and curvature by randomly generating the subsets in Definitions \ref{sratio} and \ref{curvature}. In particular, for each set of sampled subsets, we determined the respective largest and smallest values of $\gamma$ and $\alpha$ so that the inequalities in these definitions are satisfied. We never encountered sets that violated the submodularity inequality, indicating that although the set functions are not submodular in general, they may often be close to submodular empirically in typical instances. Similarly, the curvature ratios varied between zero and one, but on average were very close to the ideal value of zero. Representative results based on 5000 subset samples for each network type with $n=50$ using the trace inverse metric are shown in Table \ref{sample-table}. Of course the theoretical bounds provide the only hard performance guarantees, but these empirical values may be stronger indicators of empirical performance of the greedy algorithm. In all examples, to find a Gramian whose inverse is well-defined, we included a small identity matrix in the Lyapunov equation (i.e., we solved $AW_S + W_SA^T + B_S B_S^T + \epsilon I$). This is consistent with the assumption needed for deriving the bounds, namely, having  a set of existing actuators that provide controllability. 

To provide evidence supporting the empirical effectiveness of the greedy algorithm despite lack of submodularity of minimum eigenvalue and trace inverse Gramian, we compared greedy results with globally optimal results for problems small enough to allow brute force search. We randomly generated 500 instances each for several types of networks, including random stable (via Matlab's rss function), Erd{\H{o}}s-R{\'e}nyi , and Barab{\'a}si-Albert, with $n=16$ and $k=4$. We find that on average the greedy algorithm achieves over 90\% of the globally optimal value across all networks for both the minimum eigenvalue and trace inverse metrics, and in many cases recovers a globally optimal actuator selection. This significantly outperforms even the worst-case guarantee for submodular functions (of $\sim 63 \%$), and supports the boosts suggested by the empirical estimates of the submodularity ratio and curvature in Table 1. Note that for a submodular function $f$, a small curvature improves the performance guarantee of the greedy algorithm via Theorem \ref{nonsubmodthm}, e.g. $\alpha = 0.01$ would ensure $99.5 \%$  optimality. ). An interesting future work is to understand probability of these metrics being submodular on certain classes of random graphs. 

\begin{table}[t]
\caption{Empirical estimates of submodularity ratio and curvature for various networks, based on 5000 random subset pairs from Definitions \ref{sratio} and \ref{curvature}.}
\label{sample-table}
\vskip 0.15in
\begin{center}
\begin{small}
\begin{sc}
\begin{tabular}{lcccr}
\toprule
Network & $\gamma_{\text{emp}}$ & $\alpha_{\text{emp}}$ (range, average) \\
\midrule
Erd{\H{o}}s-R{\'e}nyi    & 1 & [0, 0.72], \ 0.010  \\
Barab{\'a}si-Albert & 1 & [0, 0.66], \ 0.009 \\
L-shaped mesh    & 1 & [0, 0.99], \ 0.007 \\
\bottomrule
\end{tabular}
\end{sc}
\end{small}
\end{center}
\vskip -0.1in
\end{table}

A difficulty in  implementing the algorithm as well as evaluating and interpreting the bounds and empirical estimates of submodularity ratio and curvature is that even for these moderately sized networks, the Gramian usually has several very small eigenvalues corresponding to state space directions that require large input energy to achieve state transfer. This observation is  theoretically supported by the fundamental limitations discussed \cite{Pasqualetti2014c}. The pseudo-inverse or minimum non-zero eigenvalue can be used as alternatives, but these values are highly sensitive to an arbitrary threshold defining which eigenvalues are considered numerically zero. Appropriately deciding a controllability metric for large networks requires consideration of the application context.


\section{Conclusions}
We  derived bounds on the submodularity ratio and curvature for two important non-submodular set functions related to the controllability Gramian. These bounds justify the use of the greedy approach beyond a heuristic, for large-scale network design problems corresponding to the Gramian. In  simulations, we observed that the bounds derived might be very conservative. We are currently investigating reducing this conservatism for specific classes of networks. A major assumption in our work is that a base set of actuators for controllability exist.  Currently, we are working on deriving alternative formulations of the problem that bypass this restrictive assumption, similar in spirit to \cite{tzoumas2016minimal} and  can also provide a meaningful metric in the cases where several eigenvalues of the Gramian are near zero as proven in \cite{Pasqualetti2014c}.

\bibliographystyle{IEEEtran}
\bibliography{refs.bib}

\begin{thebibliography}{10}
\providecommand{\url}[1]{#1}
\csname url@samestyle\endcsname
\providecommand{\newblock}{\relax}
\providecommand{\bibinfo}[2]{#2}
\providecommand{\BIBentrySTDinterwordspacing}{\spaceskip=0pt\relax}
\providecommand{\BIBentryALTinterwordstretchfactor}{4}
\providecommand{\BIBentryALTinterwordspacing}{\spaceskip=\fontdimen2\font plus
\BIBentryALTinterwordstretchfactor\fontdimen3\font minus
  \fontdimen4\font\relax}
\providecommand{\BIBforeignlanguage}[2]{{%
\expandafter\ifx\csname l@#1\endcsname\relax
\typeout{** WARNING: IEEEtran.bst: No hyphenation pattern has been}%
\typeout{** loaded for the language `#1'. Using the pattern for}%
\typeout{** the default language instead.}%
\else
\language=\csname l@#1\endcsname
\fi
#2}}
\providecommand{\BIBdecl}{\relax}
\BIBdecl

\bibitem{liu2011controllability}
Y.-Y. Liu, J.-J. Slotine, and A.-L. Barab{\'a}si, ``Controllability of complex
  networks,'' \emph{Nature}, vol. 473, no. 7346, pp. 167--173, 2011.

\bibitem{nepusz2012controlling}
T.~Nepusz and T.~Vicsek, ``Controlling edge dynamics in complex networks,''
  \emph{Nature Physics}, vol.~8, no.~7, pp. 568--573, 2012.

\bibitem{ruths2014control}
J.~Ruths and D.~Ruths, ``Control profiles of complex networks,''
  \emph{Science}, vol. 343, no. 6177, pp. 1373--1376, 2014.

\bibitem{olshevsky2014minimal}
A.~Olshevsky, ``Minimal controllability problems,'' \emph{IEEE Transactions on
  Control of Network Systems}, vol.~1, no.~3, pp. 249--258, 2014.

\bibitem{pequito2016}
S.~Pequito, S.~Kar, and A.~Aguiar, ``A framework for structural input/output
  and control configuration selection in large-scale systems,'' \emph{IEEE
  Transactions on Automatic Control}, vol.~61, no.~2, pp. 303--318, 2016.

\bibitem{Pasqualetti2014c}
F.~Pasqualetti, S.~Zampieri, and F.~Bullo, ``Controllability metrics,
  limitations and algorithms for complex networks,'' \emph{IEEE Transactions on
  Control of Network Systems}, vol.~1, no.~1, pp. 40--52, 2014.

\bibitem{summers2014optimal}
T.~H. Summers and J.~Lygeros, ``Optimal sensor and actuator placement in
  complex dynamical networks,'' \emph{IFAC Proceedings Volumes}, vol.~47,
  no.~3, pp. 3784--3789, 2014.

\bibitem{summers2014submodularity}
T.~Summers, F.~Cortesi, and J.~Lygeros, ``On submodularity and controllability
  in complex dynamical networks,'' \emph{IEEE Transactions on Control of
  Network Systems}, vol.~3, no.~1, pp. 91--101, 2016.

\bibitem{tzoumas2016}
V.~Tzoumas, M.~A. Rahimian, G.~Pappas, and A.~Jadbabaie, ``Minimal actuator
  placement with bounds on control effort,'' \emph{IEEE Transactions on Control
  of Network Systems}, vol.~3, no.~1, pp. 67--78, 2016.

\bibitem{yannature2015}
G.~Yan, G.~Tsekenis, B.~Barzel, J.-J. Slotine, Y.-Y. Liu, and A.-L. Barab\'asi,
  ``Spectrum of controlling and observing complex networks,'' \emph{Nature
  Physics}, vol.~11, pp. 779--786, 2015.

\bibitem{zhao2016scheduling}
Y.~Zhao, F.~Pasqualetti, and J.~Cort{\'e}s, ``Scheduling of control nodes for
  improved network controllability,'' in \emph{Proc. IEEE Conference on
  Decision and Control}, 2016, pp. 1859--1864.

\bibitem{nozari2016time}
E.~Nozari, F.~Pasqualetti, and J.~Cortes, ``Time-varying actuator scheduling in
  complex networks,'' \emph{arXiv preprint arXiv:1611.06485}, 2016.

\bibitem{Polyak-LMI_sparse_fb}
B.~Polyak, M.~Khlebnikov, and P.~Shcherbakov, ``An {LMI} approach to structured
  sparse feedback design in linear control systems,'' in \emph{Proc. European
  Control Conference}, July 2013, pp. 833--838.

\bibitem{Dhingra2014}
N.~K. Dhingra, M.~R. Jovanovi{\'c}, and Z.-Q. Luo, ``An {ADMM} algorithm for
  optimal sensor and actuator selection,'' in \emph{Proc. IEEE Conference on
  Decision and Control}, 2014, pp. 4039--4044.

\bibitem{munz2014sensor}
U.~M{\"u}nz, M.~Pfister, and P.~Wolfrum, ``Sensor and actuator placement for
  linear systems based on {$H_2$} and {$H_\infty$} optimization,'' \emph{IEEE
  Transactions on Automatic Control}, vol.~59, no.~11, pp. 2984--2989, 2014.

\bibitem{Argha2016}
A.~Argha, S.~W. Su, and A.~Savkin, ``Optimal actuator/sensor selection through
  dynamic output feedback,'' in \emph{Proc. IEEE 55th Conference on Decision
  and Control}, Dec. 2016, pp. 3624--3629.

\bibitem{summers2016actuator}
T.~Summers, ``Actuator placement in networks using optimal control performance
  metrics,'' in \emph{Proc. IEEE Conference on Decision and Control}, 2016, pp.
  2703--2708.

\bibitem{zhang2017sensor}
H.~Zhang, R.~Ayoub, and S.~Sundaram, ``Sensor selection for {K}alman filtering
  of linear dynamical systems: Complexity, limitations and greedy algorithms,''
  \emph{Automatica}, vol.~78, pp. 202--210, 2017.

\bibitem{Taylor2017}
J.~A. Taylor, N.~Luangsomboon, and D.~Fooladivanda, ``Allocating sensors and
  actuators via optimal estimation and control,'' \emph{IEEE Transactions on
  Control Systems Technology}, vol.~25, no.~3, pp. 1060--1067, May 2017.

\bibitem{Taha2017}
A.~Taha, N.~Nugroho, T.~Summers, and N.~Gatsis, ``Time-varying sensor and
  actuator selection for uncertain cyber-physical systems,'' \emph{arXiv
  preprint arXiv:1708.07912}, 2017.

\bibitem{clark2016submodularity}
A.~Clark, B.~Alomair, L.~Bushnell, and R.~Poovendran, \emph{Submodularity in
  dynamics and control of networked systems}.\hskip 1em plus 0.5em minus
  0.4em\relax Springer, 2016.

\bibitem{summers2017information}
T.~Summers, C.~Li, and M.~Kamgarpour, ``Information structure design in team
  decision problems,'' in \emph{IFAC World Congress, Toulouse, France}, 2017,
  pp. 2530--2535.

\bibitem{summers2016convex}
T.~Summers and I.~Shames, ``Convex relaxations and gramian rank constraints for
  sensor and actuator selection in networks,'' in \emph{IEEE Multi-conference
  on Systems and Control}.\hskip 1em plus 0.5em minus 0.4em\relax IEEE, 2016,
  pp. 1--6.

\bibitem{bian2017guarantees}
A.~A. Bian, J.~M. Buhmann, A.~Krause, and S.~Tschiatschek, ``Guarantees for
  greedy maximization of non-submodular functions with applications,''
  \emph{arXiv preprint arXiv:1703.02100}, 2017.

\bibitem{summers2017kamgarpour}
T.~Summers and M.~Kamgarpour, ``Performance guarantees for greedy maximization
  of non-submodular set functions in systems and control,'' \emph{arXiv
  preprint arXiv:1712.04122}, 2017.

\bibitem{gupta2018approximate}
G.~Gupta, S.~Pequito, and P.~Bogdan, ``Approximate submodular functions and
  performance guarantees,'' \emph{arXiv preprint arXiv:1806.06323}, 2018.

\bibitem{tzoumas2018control}
V.~Tzoumas, L.~Carlone, G.~J. Pappas, and A.~Jadbabaie, ``Control and sensing
  co-design,'' \emph{arXiv preprint arXiv:1802.08376}, 2018.

\bibitem{nemhauser1978analysis}
G.~L. Nemhauser, L.~A. Wolsey, and M.~L. Fisher, ``An analysis of
  approximations for maximizing submodular set functions?i,''
  \emph{Mathematical Programming}, vol.~14, no.~1, pp. 265--294, 1978.

\bibitem{summers2017submodularitycorrection}
T.~Summers, F.~Cortesi, and J.~Lygeros, ``Correction to ``on submodularity and
  controllability in complex dynamical networks'','' \emph{IEEE Transactions on
  Control of Network Systems}, to appear, 2017.

\bibitem{clark2017submodularity}
A.~Clark, B.~Alomair, L.~Bushnell, and R.~Poovendran, ``Submodularity in input
  node selection for networked linear systems: Efficient algorithms for
  performance and controllability,'' \emph{IEEE Control Systems}, vol.~37,
  no.~6, pp. 52--74, 2017.

\bibitem{weyl1912asymptotische}
H.~Weyl, ``Das asymptotische verteilungsgesetz der eigenwerte linearer
  partieller differentialgleichungen (mit einer anwendung auf die theorie der
  hohlraumstrahlung),'' \emph{Mathematische Annalen}, vol.~71, no.~4, pp.
  441--479, 1912.

\bibitem{tzoumas2016minimal}
V.~Tzoumas, M.~A. Rahimian, G.~J. Pappas, and A.~Jadbabaie, ``Minimal actuator
  placement with bounds on control effort,'' \emph{IEEE Transactions on Control
  of Network Systems}, vol.~3, no.~1, pp. 67--78, 2016.

\bibitem{barabasi1999emergence}
A.-L. Barab{\'a}si and R.~Albert, ``Emergence of scaling in random networks,''
  \emph{Science}, vol. 286, no. 5439, pp. 509--512, 1999.

\bibitem{erdHos1959random}
P.~Erd{\H{o}}s and A.~R{\'e}nyi, ``On random graphs,'' \emph{Publicationes
  Mathematicae Debrecen}, vol.~6, pp. 290--297, 1959.

\end{thebibliography}

\end{document}